\documentclass{amsart}
\usepackage[english]{babel}
\usepackage[latin1]{inputenc}
\usepackage[dvips,final]{graphics}
\usepackage{amsmath,amsfonts,amssymb,amsthm,amscd,array,stmaryrd,mathrsfs, mathdots, epigraph}
\usepackage{pstricks}
 \usepackage[all]{xy}
 \usepackage{url}
\usepackage{multirow, blkarray}
\usepackage{booktabs}
\usepackage{textcomp}
 \usepackage[final]{epsfig}
 \usepackage{color}
\theoremstyle{plain}
\newtheorem{lem}{Lemma}[section]
\newtheorem{thm}[lem]{Theorem}
\newtheorem{cor}[lem]{Corollary}
\newtheorem{prop}[lem]{Proposition}
\newtheorem{defn}[lem]{Definition}

\theoremstyle{definition}
\newtheorem{rem}[lem]{Remark}
\newtheorem{ex}[lem]{Example}

\newcommand{\R}{\mathbb{R}}
\newcommand{\Z}{\mathbb{Z}}

\newcommand{\N}{\mathbb{N}}

\newcommand{\Qc}{\mathcal{Q}}

\def\e{\varepsilon}

\begin{document}

\title{Dual numbers, weighted quivers, and 
extended Somos and Gale-Robinson sequences}

\author{Valentin Ovsienko} \author{Serge Tabachnikov}

\address{
Valentin Ovsienko,
CNRS,
Laboratoire de Math\'ematiques 
U.F.R. Sciences Exactes et Naturelles 
Moulin de la Housse - BP 1039 
51687 REIMS cedex 2,
France}

\address{
Serge Tabachnikov,
Pennsylvania State University,
Department of Mathematics,
University Park, PA 16802, USA,
}

\email{valentin.ovsienko@univ-reims.fr
tabachni@math.psu.edu}



\begin{abstract}
We investigate a general method that allows one
to construct new integer sequences extending existing ones.
We apply this method to the classic Somos-4 and Somos-5,
and the Gale-Robinson sequences,
as well as to more general class of sequences introduced by Fordy and Marsh,
and produce a great number of new sequences.
The method is based on the notion of ``weighted quiver'',
a quiver with a $\Z$-valued function on the set of vertices
that obeys very special rules of mutation.
\end{abstract}

\maketitle

\thispagestyle{empty}

\section{Introduction}

A  {\it dual number} is a pair of real numbers $a$ and $b$, written in the form
$a+b\e$, subject to the condition that $\e^2=0$. 
Dual numbers form a commutative algebra,
they were introduced by Clifford in 1873, and since then have found applications
in geometry and mathematical physics. For example, according to E. Study, the 
space of oriented lines in $\R^3$ is the unit sphere in the 3-dimensional space over dual numbers \cite{Stu}. 
 Surprisingly, dual numbers are not frequent guests in number theory
 and combinatorics. In this paper, we will use dual numbers 
 to construct a large family of integer sequences.

Let $(a_n)_{n\in\N}$ be an integer sequence defined by some recurrence
and initial conditions.
We will consider a pair of sequences,
$(a_n)_{n\in\N}$ and $(b_n)_{n\in\N}$, organized as a sequence of dual numbers:
\begin{equation}
\label{DualN}
A_n:=a_n+b_n\e.
\end{equation} 
We assume that $(A_n)_{n\in\N}$ satisfies
either exactly the same recurrence as $(a_n)_{n\in\N}$,
or its certain deformation.
We furthermore choose the same
initial conditions for $(a_n)_{n\in\N}$ and arbitrary initial conditions for $(b_n)_{n\in\N}$.
We show that, in many interesting cases, the
sequence~$(b_n)_{n\in\N}$, defined in this way, is an integer sequence.
This method was suggested in~\cite{SCl}, and tested on the  
Somos-4 sequence, producing new integer sequences.

\subsection{Somos, Gale-Robinson and beyond}\label{SeqSect}
Let us give a brief overview of the integer sequences
that we will consider.

The Somos-4 and Somos-5 sequences are
the sequences of integers, $(a_n)_{n\in\N}$, defined by the recurrences:
$$
a_{n+4}a_n=a_{n+3}a_{n+1}+a_{n+2}^2
\qquad\hbox{and}\qquad
a_{n+5}a_n=a_{n+4}a_{n+1}+a_{n+3}a_{n+2},
$$
and the initial conditions: $a_1=a_2=a_3=a_4=1$ 
and $a_1=a_2=a_3=a_4=a_5=1$, respectively.
These sequences were discovered by Michael Somos in 80's,
and they are discrete analogs of elliptic functions.
Their integrality, observed and conjectured by Somos, was
later proved by several authors; for a historic account see~\cite{Gal}.

A more general class of integer sequences generalizing the
Somos sequences, called the three term Gale-Robinson sequences,
are defined by the recurrences:
\begin{equation}
\label{GRGen}
a_{n+N}a_n=a_{n+N-r}a_{n+r}+a_{n+N-s}a_{n+s},
\end{equation}
where $1\leq{}r<s\leq{}\frac{N}{2}$,
and the initial conditions $(a_1,\ldots,a_N)=(1,\ldots,1)$.
Their integrality was proved in~\cite{FZ2}.
A combinatorial proof was then given in~\cite{BPW} and~\cite{Spe},
a proof that explicitly uses quiver mutations and cluster algebras
was presented in~\cite{FM}.

A large class of integer sequences generalizing those of Gale-Robinson
was introduced in~\cite{FM} (see also~\cite{Mar}).
The recurrence is of the general form
\begin{equation}
\label{FMRec}
a_{n+N}a_n=P(a_{n+N-1},\ldots,a_{n+1}),
\end{equation}
where the polynomial~$P$ is a sum of two monomials:
$P=P_1+P_2$.
The initial conditions are those of Gale-Robinson.
A simple example of such sequences is:
$$
a_{n+4}a_n=a_{n+3}^pa_{n+1}^p+a_{n+2}^q,
$$
with arbitrary positive integers $p,q$.
Note that this particular sequence was already considered in~\cite{Gal},
where their integrality was clamed.

The method of~\cite{FM} is based on the Fomin-Zelevinsky
{\it Laurent phenomenon}~\cite{FZ2} and on the notion of
{\it period 1 quiver}, i.e.,
a quiver that rotates under mutations.

\subsection{Extensions}\label{WExSect}

The same arguments show that, for $A_n=a_n+b_n\e$ satisfying the recurrence
$$
A_{n+N}A_n=P(A_{n+N-1},\ldots,A_{n+1}),
$$
where $P$ is as in~\cite{FM}, and
$(a_1,\ldots,a_N)=(1,\ldots,1)$,
the sequence~$(b_n)_{n\in\N}$ is integer for an arbitrary choice of integral initial conditions
$(b_1,\ldots,b_N)$.
In fact, this is a direct consequence of the classic
Laurent phenomenon.

Note that the recurrence for~$(b_n)_{n\in\N}$ obtained in this way, is nothing other than
the {\it linearization} of~(\ref{FMRec}).
This linearization procedure already provides a large number of new integer sequences.

We will also consider the recurrences of the following form:
\begin{equation}
\label{GenExtF}
A_{n+N}A_n=P_1+P_2\left(1+w\e\right)
\qquad\hbox{and}\qquad
A_{n+4}A_n=P_1\left(1+w\e\right)+P_2,
\end{equation}
where $w$ is an arbitrary integer,
and $P_i$ stands for $P_i(A_{n+N-1},\ldots,A_{n+1})$, $i=1,2$.
In this case,~$(b_n)_{n\in\N}$ satisfies a {\it non-linear} recurrence.
More precisely, the recurrence for~$(b_n)_{n\in\N}$ is given by an affine function
with polynomial in~$(a_n)_{n\in\N}$ coefficients.

We show, in particular, that these non-linear extensions of the Gale-Robinson sequences
are always integer.
We give a sufficient condition for the corresponding period~1 quiver
that guaranties that the extensions of the form~(\ref{GenExtF}) 
generates an integer sequence $(b_n)$.
Integrality of the sequences defined by~(\ref{GenExtF})
is a consequence of a version of the Laurent phenomenon
proved in~\cite{SCl} for cluster superalgebras.

\subsection{Weighted quivers}\label{WeSect}
Our main tool is what we call
a weighted quiver.
This is a usual quiver~$\Qc$
(without $1$- or $2$-cycles), together with a function
$w:\Qc_0\to\Z$ on the set of vertices.
Quiver mutations for such weight functions are defined as follows.
Label the vertices by positive integers $1,\ldots,n$,
the mutation~$\mu_k$ at $k$th vertex sends $w$ to
the new function~$\mu_k(w)$ defined by:
$$
\begin{array}{rcll}
\mu_k(w)(i) &=& w(i)+[b_{ki}]_+w(k), & i\not={}k,\\[4pt]
\mu_k(w)(k) &=& -w(k),&
\end{array}
$$
where $[b_{ki}]_+$ is the number of arrows from the vertex
$k$ to the vertex $i$, and if the vertices are oriented from $i$ to $k$, then
$[b_{ki}]_+=0$.
The exchange relations are also modified.

Let us mention that the mutation rule of the weight function
that we use (but not the exchange relations), have already been introduced 
by several authors; see formula~(2.3) of~\cite{GHK} and~\cite{Gal1},\cite{Gal2}.
It would be interesting to investigate the relations of our work with these papers.

We classify the period~$1$ quivers (in the sense of Fordy-Marsh~\cite{FM})
that have a period~$1$ weight function.

\section{Extensions of the Somos-$4$ sequence}\label{S4Sect}

We start with the extensions of the Somos-4 sequence,
briefly considered in~\cite{SCl}.
The goal of this section is to show that
the new sequences arising in this way have nice properties.
This section is based on a computer program written by Michael Somos,
to whom we are most grateful.
It can be considered as a motivation for the rest of the paper.

\subsection{Linearization}

Consider first the recurrence
\begin{equation}
\label{S4ExtL}
A_{n+4}A_n=A_{n+3}A_{n+1}+A_{n+2}^2,
\end{equation}
where $A_n=a_n+b_n\e$ as in~(\ref{DualN}).
The sequence $(a_n)_{n\in\N}$ is then the Somos-4 sequence:
$$
a_n = 1, 1, 1, 1, 2, 3, 7, 23, 59, 314, 1529, 8209, 83313, 620297, 7869898, \ldots
$$
while the sequence $(b_n)_{n\in\N}$ satisfies:
$$
b_{n+4}a_n+b_na_{n+4}=
a_{n+1}b_{n+3}+
2a_{n+2}b_{n+2}+a_{n+3}b_{n+1},
$$
which is the {\it linearization} of
the Somos-4 recurrence.

The space of solutions of the linearized system is a four-dimensional vector space.
Every sequence $(b_n)_{n\in\N}$ satisfying this recurrence is a linear
combination of the sequences with one of the following initial
conditions:
$$
(b_1,b_2,b_3,b_4)\quad=\quad
(1,0,0,0),
\quad
(0,1,0,0),
\quad
(0,0,1,0),
\quad\hbox{or}\quad
(0,0,0,1).
$$
Let us denote these sequences by
$(b^1_n)_{n\in\N},(b^2_n)_{n\in\N},(b^3_n)_{n\in\N},(b^4_n)_{n\in\N}$,
respectively.

We have the following properties:
\begin{enumerate}
\item
The four  sequences $(b^1_n)_{n\in\N},(b^2_n)_{n\in\N},(b^3_n)_{n\in\N},(b^4_n)_{n\in\N}$ are integer.
\item
The Somos-4 sequence is their sum, i.e.,
$$
a_n=b^1_n+b^2_n+b^3_n+b^4_n,
$$ 
for every $n$.
\end{enumerate}

Integrality of each of the four sequences $(b_n^i)_{n\in\N}$
follows from the classic Laurent phenomenon.
Indeed, $A_n$ with $n\geq5$ is a Laurent polynomial
in $A_1,A_2,A_3,A_4$.
However, the denominators of this Laurent polynomial
are monomials in $a_1,a_2,a_3,a_4$,
while $b_1,b_2,b_3,b_4$ enter (linearly) into the numerators.
Indeed, since $\e^2=0$, one has
$$
\frac{1}{a+b\e}=
\frac{1}{a}-\frac{b}{a^2}\e.
$$

Property (2) follows from the fact that the
sequences $(b_n)_{n\in\N}$ satisfying the recurrence~(\ref{S4ExtL})
form a vector space, and the initial condition of the sum
$(b^1_n+b^2_n+b^3_n+b^4_n)_{n\in\N}$ is precisely $(1,1,1,1)$,
i.e., that of $(a_n)_{n\in\N}$.
Property (2) means that the sequences 
$(b^1_n)_{n\in\N},(b^2_n)_{n\in\N},(b^3_n)_{n\in\N},(b^4_n)_{n\in\N}$ 
provide a canonical
way to decompose the Somos-4 sequence with respect to the initial conditions.

Note, however, that the sequences $(b^1_n)_{n\in\N}$ and $(b^2_n)_{n\in\N}$
are not positive, the first values being:
$$
\begin{array}{rcllllllllllllll}
b^1_n &= &1,&0,&0,&0, &-2, &-2, &-10,& -46, &-103, &-933,& -4681,& -27912, &-375536,&\ldots\\[2pt]
b^2_n &= &0,&1,&0,&0, &1,&-2,&2,&-1,&-40,&140,&-696,&-265,&38478,&\ldots\\[2pt]
b^3_n &= &0,&0,&1,&0, &2,&4,&5,&48,&94,&635,&4732,&18594,&299835,&\ldots\\[2pt]
b^4_n &= &0,&0,&0,&1,&1,&3,&10,&22,&108,&472,&2174,&17792,&120536,&\ldots
\end{array}
$$
It would be interesting to understand the properties of these sequences.
Computer calculations show that $b^1_n<0$, for every $n\leq100$.
The sequence $(b^2_n)_{n\in\N}$ becomes positive for $n\geq15$
(this checked for~$n\leq100$).
Conjecturally, the sequence $(b^3_n)_{n\in\N}$ is positive.
The sequence $(b^4_n)_{n\in\N}$ is positive for~$n\leq25$,
however, $b^4_n<0$, for $26\leq{}n\leq100$.

\subsection{Two non-linear extensions}

Consider now the
extensions of the Somos-4 sequence, satisfying the recurrences:
$$
\begin{array}{rcl}
A_{n+4}A_n &= &A_{n+3}A_{n+1}+A_{n+2}^2\left(1+w\e\right)\\[4pt]
A_{n+4}A_n &=& A_{n+3}A_{n+1}\left(1+w\e\right)+A_{n+2}^2,
\end{array}
$$
where $w\in\Z$ is an arbitrary (fixed) integer.
The corresponding recurrences for $(b_n)_{n\in\N}$ are non-linear:
\begin{eqnarray}
\label{S4ExtNL1}
b_{n+4}a_n&=&a_{n+1}b_{n+3}+
2a_{n+2}b_{n+2}+a_{n+3}b_{n+1}-b_na_{n+4}+w\,a_{n+2}^2,\\[2pt]
\label{S4ExtNL2}
b_{n+4}a_n&=&a_{n+1}b_{n+3}+
2a_{n+2}b_{n+2}+a_{n+3}b_{n+1}-b_na_{n+4}+w\,a_{n+3}a_{n+1},
\end{eqnarray}
respectively, where $(a_n)_{n\in\N}$ is the initial Somos-4 sequence.
It was proved in~\cite{SCl} that, for any choice of integral initial conditions, 
the sequence~$(b_n)_{n\in\N}$ satisfying~(\ref{S4ExtNL1}) or~(\ref{S4ExtNL2}) is integer.

The choice of zero initial conditions:
\begin{equation}
\label{ZeroIC}
(b_1,b_2,b_3,b_4)=(0,0,0,0)
\end{equation}
is now the most natural one.
Indeed, any solution $(b_n)$ 
of each of the recurrences~(\ref{S4ExtNL1}) and~(\ref{S4ExtNL2}) is the sum of the solution
with the initial condition~(\ref{ZeroIC})
and a solution of the linear recurrence~(\ref{S4ExtL}).
In other words, the solutions of~(\ref{S4ExtNL1}) and~(\ref{S4ExtNL2})
form an affine subspace.

Choosing $w=1$,
the sequences $(b_n)_{n\in\N}$ defined by the recurrences~(\ref{S4ExtNL1}) and~(\ref{S4ExtNL2})
and zero initial conditions~(\ref{ZeroIC}) start as follows:
$$
\begin{array}{rcl}
b_n &=&
0,0,0,0,1,2,10,48,160,1273,7346,51394,645078,5477318,87284761\ldots\\[4pt]
b_n &=&
0,0,0,0,1,3,10,59,198,1387,9389,57983,752301,6851887,97297759\ldots
\end{array}
$$
respectively.
We conjecture the positivity of these sequences, and we checked it for $n\leq100$.
Furthermore, conjecturally, both of the above sequences grow faster than $(a_n)_{n\in\N}$.

\section{Weighted quivers, mutations and exchange relations}\label{WeightSect}

In this section, we introduce the notion of weighted quiver.
We then describe the mutation rules of such objects,
extending the usual mutation rules of quivers.
We also describe the modified exchange relations for 
weighted quivers,
and formulate the corresponding Laurent phenomenon.

The notion of weighted quiver
is equivalent to the ``simplified version'' of cluster superalgebra
with two odd variables~\cite{SCl}.

\subsection{Mutation rules}
A quiver $\Qc$ is an oriented finite graph with vertex set 
$\Qc_0$ and the set of arrows $\Qc_1$.
Usually, the vertices of $\Qc$ will be labeled by the
letters $\{x_1,\ldots,x_N\}$, where $N=|\Qc_0|$,
considered as formal variables.

In~\cite{FZ1}, Fomin and Zelevinsky
defined the rules of mutation of a quiver, under the assumption
that $\Qc$ has no $1$-cycles and $2$-cycles.
The structure of $\Qc$ can then be represented as an $N\times{}N$-skew-symmetric
matrix $(b_{ij})$, where
$b_{ij}$ is the number of arrows between the vertices $x_i$ and $x_j$.
Note that the sign of $b_{ij}$ depends on the orientation:
$b_{ij}>0$ if the arrows are oriented from~$x_i$ to~$x_j$ and negative otherwise.

The mutation of the quiver $\mu_k:\Qc\to\Qc'$
at vertex $x_k$ is defined by the following three rules:
\begin{itemize}
\item for every path $(x_i\rightarrow x_k \rightarrow x_j)$ in $\Qc$, add an arrow $(x_i\rightarrow x_j)$;
\item reverse all the arrows incident with $x_k$;
\item remove all 2-cycles created by the first rule.
\end{itemize}

\begin{defn}
\label{TheMainDef}
We call a {\it weighted quiver} a quiver $\Qc$ with a function
$$
w:\Qc_0\to\Z.
$$
The function $w$ associates to every variable $x_i$ its {\it weight},
$w_i:=w(x_i)$.

The mutation $\mu_k(w)$ of the weight function $w$ is performed according to the
following two rules:
\begin{enumerate}
\item for every arrow $(x_k\rightarrow x_i)$, change the value $w_i$ to $w_i+w_k$.
In other words, the new weight function $\mu_k(w)$ is defined by
$$
\mu_k(w)(x_i):=w_i+[b_{ki}]_+w_k,
$$
where 
$$
[b_{ki}]_+=\left\{
\begin{array}{l}
b_{ki},\; \hbox{if}\; b_{ki}\geq0,\\
0,\;\hbox{otherwise}
\end{array}
\right.
$$ 
for all $i\not=k$;
\item reverse the sign of $w_k$, {\it i.e.}, 
$$
\mu_k(w)(x_k):=-w_k.
$$
\end{enumerate}
\end{defn}

\subsection{Exchange relations}
Recall that the mutation $\mu_k$ of the quiver $\Qc$
replaces the variable $x_k$ by the new 
function~$x_k'$ defined by the formula:
$$
x_kx'_k=\prod\limits_{\substack{x_k\to x_j}}\;x_j
 \quad+\quad 
\prod\limits_{\substack{x_i\to x_k}}\; x_i,
$$
where the products are taken over the set of arrows
$(x_i\rightarrow x_k)\in\Qc_1$ and $(x_k\to x_j)\in\Qc_1$, respectively
(with fixed $k$).
The above formula is
called the \textit{exchange relation}.
The {\it Laurent phenomenon}, proved in~\cite{FZ1},
states that every (rational) function obtained by a series of
mutations is actually a Laurent polynomial in the initial variables
$\{x_1,\ldots,x_N\}$.

Given a weighted quiver $(\Qc,w)$,
we assume that the vertices are labeled by the variables 
$\{X_1,\ldots,X_N\}$ written as dual numbers:
$$
X_i=x_i+y_i\e,
$$
where $x_i$ and $y_i$ are the usual commuting variables.
The exchange relations are defined as follows.

\begin{defn}
The mutation $\mu_k$ of $(\Qc,w)$ replaces the variable~$X_k$ by a new variable, $X_k'$,
defined by the formula
\begin{equation}
\label{Mute}
X_kX_k'=
\prod\limits_{\substack{X_k\to X_j }}X_j
\quad+\quad
\left(1+
w_k\e\right)
\prod\limits_{\substack{X_i\to X_k }}X_i;
\end{equation}
the other variables remain unchanged.
\end{defn}

This is a particular case of the exchange relations defined in~\cite{SCl}.

\subsection{Laurent phenomenon}

The following version of Laurent phenomenon is proved in~\cite{SCl}
(this is the simplest case of Theorem~$1$).

\begin{thm}
\label{LeurThm}
For every weighted quiver~$(\Qc, w)$, all the 
functions $X_k', X_k'',\ldots$,
obtained recurrently by any series of consecutive mutations,
are Laurent polynomials
in the initial coordinates $\{X_1,\ldots,X_N\}$.
\end{thm}

\begin{rem}
Note that Laurentness in $\{X_1,\ldots,X_N\}$ means
that the denominators are monomials in the variables $\{x_1,\ldots,x_N\}$,
while the variables $\{y_1,\ldots,y_N\}$ enter (linearly) into the numerators.
\end{rem}

\section{Pperiod~$1$ weighted quivers}\label{POnetSect}
Period~$1$ quivers were introduced and classified in~\cite{FM}.
These are quivers for which there exists a vertex such that
the quiver rotates under the mutation at this vertex.
More precisely, a period~$1$ quiver remains unchanged after the
mutation composed with the shift of the indices of the vertices $i\to{}i-1$.

In this section, we answer the question which period~$1$ quivers
have period~$1$ weight functions. 
A period~$1$ quiver equipped with a period~$1$ weight function
guarantees the integrality of sequence~$(b_n)_{n\in\N}$
defined by~(\ref{GenExtF}).

\subsection{Examples}
We start with simple examples.

\begin{ex}
\label{FiEx}
a) Consider the following weighted quiver with three vertices.
After mutation at~$x_1$ the quiver rotates, together with the weight function:
$$
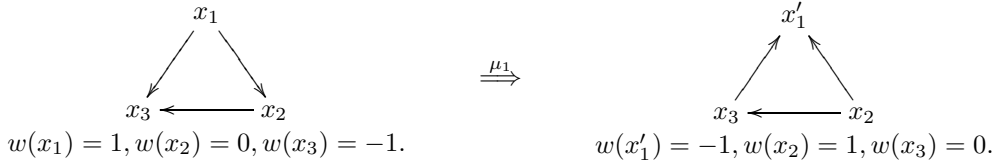
\begin{array}{c}
 \xymatrix @!0 @R=1.3cm @C=0.9cm
 {
&x_1\ar@{->}[ld]\ar@{->}[rd]&\\
x_3&&x_2\ar@{->}[ll]
}
\\
w(x_1) = 1,
w(x_2) = 0,
w(x_3) = -1.
\end{array}
\qquad
\stackrel{\mu_1}{\Longrightarrow}
\qquad
\begin{array}{c}
 \xymatrix @!0 @R=1.3cm @C=0.9cm
 {
&x'_1\ar@{<-}[ld]\ar@{<-}[rd]&\\
x_3&&x_2\ar@{->}[ll]
}
\\
w(x'_1) = -1,
w(x_2) = 1,
w(x_3) = 0.
\end{array}
$$
We will say in such a situation, that the {\it weight function has period~$1$}.

b) On the other hand, for the quiver with the inverted orientation:
$$
 
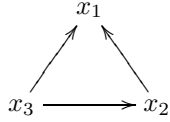
\xymatrix @!0 @R=1.3cm @C=0.9cm
 {
&x_1\ar@{<-}[ld]\ar@{<-}[rd]&\\
x_3&&x_2\ar@{<-}[ll]
}
$$
which is also of period~$1$,
there is no weight function of period~$1$. 
\end{ex}

\begin{ex}
\label{SeEx}
Similarly, for the following quivers of period~$1$ with four vertices and the opposite orientations:
$$
a) 
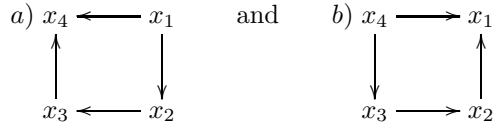
\xymatrix{
x_4\ar@{<-}[d]&
x_1\ar@{->}[d]\ar@{->}[l]\\
x_3&x_2\ar@{->}[l]
}
\qquad\hbox{and}\qquad
b) \xymatrix{
x_4\ar@{->}[d]&
x_1\ar@{<-}[d]\ar@{<-}[l]\\
x_3&x_2\ar@{<-}[l]
}
$$
the weight function
$w(x_1) = 1,w(x_2) = 0,w(x_3) =0, w(x_4) = -1,$
 has period~$1$ in the first case, and there is no such function
 in the second case.
 \end{ex}

\begin{ex}
\label{ThEx}
Another interesting example is the following Somos-4 quivers (cf.~\cite{FM} and~\cite{Mar}):
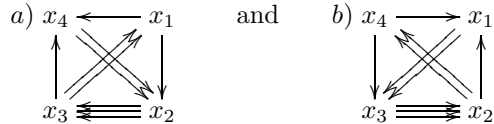
$$
a)  \xymatrix{
x_4\ar@<3pt>@{->}[rd]\ar@{->}[rd]\ar@{<-}[d]&
x_1\ar@<-3pt>@{<-}[ld]\ar@{<-}[ld]\ar@{->}[d]\ar@{->}[l]\\
x_3&x_2\ar@{->}[l]\ar@<-2pt>@{->}[l]\ar@<2pt>@{->}[l]
}
\qquad\hbox{and}\qquad
b) \xymatrix{
x_4\ar@<3pt>@{<-}[rd]\ar@{<-}[rd]\ar@{->}[d]&
x_1\ar@<-3pt>@{->}[ld]\ar@{->}[ld]\ar@{<-}[d]\ar@{<-}[l]\\
x_3&x_2\ar@{<-}[l]\ar@<-2pt>@{<-}[l]\ar@<2pt>@{<-}[l]
}
$$
Period~$1$ weight function exists for both choices of the
orientation:
$$
\begin{array}{cccc}
w(x_1) = 1,& w(x_2) = 0,& w(x_3) =0,& w(x_4) = -1,\\[2pt]
w(x_1) = 1,& w(x_2) = 1,& w(x_3) =-1,& w(x_4) = -1,
\end{array}
$$
respectively.
 \end{ex}

\subsection{Period~$1$ weight functions: criterion of existence}
Period~$1$ quivers were classified in~\cite{FM}.
The {\it primitive quiver} $P_N^{(t)}$, where $1\leq{}t\leq\frac{N}{2}$,
is a quiver with $n$ vertices and $n$ arrows, such that every vertex~$x_i$
is joined with the vertex~$x_{i+t}\;(\!\!\!\mod{N})$ where the indices are always
taken in the set $\{1,\ldots,N\}$.
The arrow is oriented from the vertex with the greater label to the vertex with the smaller label.
One then has for $P_N^{(t)}$:
$$
b_{ij}=\left\{
\begin{array}{rl}
-1, & j-i=t,\\
1, &i-j=t,\\
0, & \hbox{else}.
\end{array}
\right.
$$
For instance, the quiver considered in Examples~\ref{FiEx} b) and~\ref{SeEx} b)
are the quivers~$P_3^{(1)}$ and~$P_4^{(1)}$, respectively.

Given a quiver $\Qc$, the opposite quiver $-\Qc$ is obtained by reversing the orientation.
For example, the quivers in Examples~\ref{FiEx} a) and~\ref{SeEx} a)
are the quivers~$-P_3^{(1)}$ and~$-P_4^{(1)}$, respectively.
More generally, if $c\in\Z$, the quiver $c\,\Qc$ is obtained by 
multiplying the number of
arrows between every two vertices, $x_i$ and $x_j$ by $c$.
Finally, a sum of two quivers is obtained by superposition of their arrows.

It was proved in~\cite{FM}, that
every period~$1$ quiver can be obtained as a linear combination of
so-called primitive quivers and a correction term.
More precisely, let
$c_1,\ldots,c_r$ be arbitrary integers, 
where $r=\left[\frac{N}{2}\right]$.
A period~$1$ quiver is of the form:
\begin{equation}
\label{PerOne}
\Qc=c_1P_N^{(1)}+\cdots+c_rP_N^{(r)}+\Qc',
\end{equation}
where $\Qc'$ is a quiver with the vertices $x_2,\ldots,x_N$.
Since we will only consider the mutation at $x_1$, this ``correcting term''
$\Qc'$ will not change the exchange relations.
We thus omit the explicit form of $\Qc'$; see~\cite{FM} and~\cite{Mar}.

Let us use the notation 
$[c]_-=\left\{
\begin{array}{l}
c,c\leq0\\
0, \hbox{otherwise}.
\end{array}
\right.$

\begin{thm}
\label{Main}
Given a period~$1$ quiver $\Qc$,
there exists a period~$1$ weight function on $\Qc$
if and only if
\begin{equation}
\label{CondEq}
\begin{array}{rcll}
[c_1]_{-}+\cdots+[c_r]_{-} =1, &\hbox{if}& N &\hbox{is odd}\\[4pt]
2[c_1]_{-}+\cdots+2[c_{r-1}]_{-}+[c_r]_{-} =2, &\hbox{if}& N &\hbox{is even}.
\end{array}
\end{equation}
The period~$1$ weight function is unique up to an integer multiple.
\end{thm}

\begin{proof}
Assume that a period~$1$ quiver $\Qc$ admits a period~$1$ weight function~$w$.
By definition~\ref{TheMainDef},
the mutation at the first vertex, $\mu_1$, transforms the weight function as follows:
$$
\begin{array}{rcl}
w_1 &\mapsto & -w_1,\\[2pt]
w_i &\mapsto & w_i+[b_{1i}]_{+}w_1,
\end{array}
$$
for all $1\leq{}i\leq{}N$.
Since $w$ is of period~$1$, this implies the following system of linear equations:
$$
\begin{array}{rcl}
w_n &= & -w_1,\\[2pt]
w_1 &= & w_2+[b_{12}]_{+}w_1,\\[2pt]
w_2 &= & w_3+[b_{13}]_{+}w_1,\\
\cdots&&\\
w_{N-1} &= & w_n+[b_{1N}]_{+}w_1,
\end{array}
$$
that has (a unique) solution if and only if the following condition is satisfied:
$$
[b_{12}]_{+}+\cdots+[b_{1N}]_{+}=2.
$$

Finally, from~(\ref{PerOne}), one has $[b_{1i}]_{+}=[c_{i-1}]_{-}$,
if~$i\leq{}r$ and $[b_{1i}]_{+}=[c_{N-i+1}]_{-}$,
if~$i\geq{}r$.
The above necessary and sufficient condition for the existence of the
function $w$ then coincides with~(\ref{CondEq}).
\end{proof}

\section{Applications to integer sequences}

We apply the above constructions to integer sequences.

\subsection{The general method}\label{GenM}
Given a weighted quiver $(\Qc,w)$ of period~$1$,
by Theorem~\ref{LeurThm}, performing an infinite series
of consecutive mutations: $\mu_1,\mu_2,\ldots$, one obtains a sequence of
Laurent polynomials $(X_n)_{n\in\N}$ in the initial variables
$\{X_1,\ldots,X_N\}$.
This sequence satisfies the recurrence:
$$
X_{n+N}X_n=
\prod_{1\leq{}i\leq{}N-1}X_{n+i}^{[b_{1i}]_+}\left(1+w_1\e\right)
+\prod_{1\leq{}i\leq{}N-1}X_{n+i}^{[b_{1i}]_-}.
$$

Recall that $X_i=x_i+y_i\e$.
Choosing the initial conditions
$(x_1,\ldots,x_N):=(1,\ldots,1)$ and arbitrary integers
$(y_1,\ldots,y_N):=(b_1,\ldots,b_N)$,
one obtains a sequence $(A_n)_{n\in\N}$, where
$A_n=a_n+b_n\e$.
The constructed integer sequence $(b_n)_{n\in\N}$
is the desired extension of $(a_n)_{n\in\N}$.

Let us give further examples.

\subsection{Sequences of order 2}\label{FiboSect}

We illustrate the idea of substitution of dual numbers
into recurrences on a very simple classic example.

Consider the classic Fibonacci numbers
$(F_n)=1,\,1,\,2,\,3,\,5,\,8,\,13,\,21,\,34,\,55,\,89,\,144,\,233,\,377,\ldots$
and let us split $(F_n)$ into two subsequences:
$$
a_n:=F_{2n-1},
\qquad
\tilde a_n:=F_{2n}.
$$
Both of them satisfy quadratic recurrences that differ by a sign:
\begin{equation}
\label{ClassCass}
a_{n+2}a_n=a_{n+1}^2+1,
\qquad
\tilde a_{n+2}\tilde a_n=\tilde a_{n+1}^2-1.
\end{equation}
The above recurrences
are known as ``Cassini's identity''.
The initial conditions for these sequences are:
$(a_0,a_1)=(1,1)$ and $(\tilde a_0,\tilde a_1)=(0,1)$.

We will consider the sequences of dual numbers: 
$$
A_n:=a_n+b_n\e,
\qquad
\tilde A_n:=\tilde a_n+\tilde b_n\e,
$$
with the recurrence relations generalizing~(\ref{ClassCass}).

\subsection*{Linearization: dual Fibonacci and Lucas numbers}\label{FiboLucSect}

Suppose that $(A_n)$ and $(\tilde A_n)$ satisfy the similar recurrences:
\begin{equation}
\label{SupCass}
A_{n+2}A_n=A_{n+1}^2+1,
\qquad
\tilde A_{n+2}\tilde A_n=\tilde A_{n+1}^2-1.
\end{equation}
Equivalently, the sequences $(a_n)$ and $(\tilde a_n)$ are as above,
and $(b_n)$ and $(\tilde b_n)$ are defined by:
\begin{equation}
\label{LiF}
b_{n+2}a_n=
2b_{n+1}a_{n+1}-b_{n}a_{n+2},
\qquad
\tilde b_{n+2}\tilde a_n=
2\tilde b_{n+1}\tilde a_{n+1}-\tilde b_{n}\tilde a_{n+2}.
\end{equation}
The sequence $(b_n)$ is integer for an arbitrary choice of integral initial conditions 
$(b_0,b_1)=(p,q)$.

It turns out that the classic Lucas numbers naturally appear
in the ``dual Fibonacci'' sequences.

\begin{prop}
\label{FiboLucProp}
The sequences of odd (reps. even) Lucas numbers:
$$
b_n=L_{2n-1},
\qquad
\tilde b_n=L_{2n},
$$ 
satisfy the recurrence~(\ref{SupCass}).
\end{prop}

\begin{proof}
Using the explicit formulas for the Fibonacci and Lucas numbers
$$
F_n=\frac{\varphi^n-(-\varphi)^{-n}}{\sqrt5},
\qquad
L_n=\varphi^n+(-\varphi)^{-n},
$$
where $\varphi$ is the golden ratio,
one checks directly that the recurrences~(\ref{LiF}) are satisfied.
\end{proof}

The Lucas solutions to~(\ref{SupCass}) start as follows:
$$
\begin{array}{r|r|c|c|c|c|c|c|cc}
n&0&1&2&3&4&5&6&7&\cdots\\[2pt]
\hline
a_n&1&1&2&5&13&34&89&233&\cdots\\[2pt]
\hline
b_n&-1&1&4&11&29&76&199&521&\cdots
\end{array}
\qquad
\begin{array}{r|c|c|c|c|c|c|c|cc}
n&0&1&2&3&4&5&6&7&\cdots\\[2pt]
\hline
\tilde a_n&0&1&3&8&21&55&144&377&\cdots\\[2pt]
\hline
\tilde b_n&2&3&7&18&47&123&322&843&\cdots
\end{array}
$$

Furthermore, one has a $2$-parameter family of solutions to
the recurrence~(\ref{SupCass}):
$$
\begin{array}{r|r|c|c|c|c|c|c|cc}
n&0&1&2&3&4&5&6&7&\cdots\\[2pt]
\hline
a_n&1&1&2&5&13&34&89&233&\cdots\\[2pt]
\hline
b_n&-p&q&2p+2q&8p+3q&27p+2q&86p-10q&265p-66q&798p-277q&\cdots
\end{array}
$$
The situation is more complicated for the sequence $(\tilde b_n)_{n\in\Z}$.
Arbitrary initial conditions $(\tilde b_0,\tilde b_1)$ do not lead to
an integer sequence.
But one obtains a two-parameter family of integer sequences by
choosing the initial conditions $(\tilde b_1,\tilde b_2)=(3p,q)$
with arbitrary $p$ and $q$.
$$
\begin{array}{r|c|c|c|c|c|c|c|cc}
n&1&2&3&4&5&6&7&\cdots\\[2pt]
\hline
\tilde a_n&1&3&8&21&55&144&377&\cdots\\[2pt]
\hline
\tilde b_n&3p&q&6q-24p&25q-128p&90q-507p&300q-1778p&954q-5835p&\cdots
\end{array}
$$
Note that the sequence of coefficients of $q$ is A001871.

\subsection*{``Limping'' Fibonacci sequence}

The classic odd Fibonacci sequence $a_n=F_{2n-1}$ satisfies the first recurrence in~(\ref{ClassCass}).
It can be generated by consecutive mutations
$\mu_0,\mu_1,\mu_0,\mu_1,\ldots$ of the quiver
with two vertices and two arrows:
$$
 \xymatrix{
x_0&x_1\ar@2{->}[l]
},
$$
called the $2$-{\it Kronecker quiver}.
Clearly, there is no period~$1$ weight function $w$,
but the function $w\equiv1$ has period~$4$.

The sequence of consecutive mutations at vertices $x_1,x_2,x_1,\ldots$
then leads to the recurrence
\begin{equation}
\label{SupCassAlt}
A_{n+2}A_n=A_{n+1}^2\left(1+(-1)^{\frac{(n+1)(n+2)}{2}}\e\right)+1.
\end{equation}
More precisely, $(b_n)_{n\in\Z}$ satisfies
$$
b_{n+2}a_n+b_{n}a_{n+2}=
2b_{n+1}a_{n+1}+(-1)^{\frac{(n+1)(n+2)}{2}}a_{n+1}^2.
$$
Let us consider the initial conditions $a_0=a_1=1$ and $b_0=b_1=0$.

It turns out that the sequence $(b_n)_{n\in\N}$ also consists of Fibonacci numbers,
but this time with even indices, and taken in a surprising order:
$$
\begin{array}{r|c|c|c|c|c|c|c|cc}
n&0&1&2&3&4&5&6&7&\cdots\\[2pt]
\hline
a_n&1&1&2&5&13&34&89&233&\cdots\\[2pt]
\hline
b_n&0&0&1&8&21&21&55&377&\cdots
\end{array}
$$
More precisely, 
$$
b_n=\left\{
\begin{array}{ll}
F_{2n},&n\equiv0,3\mod4,\\
F_{2n-2},&n\equiv1,2\mod4.
\end{array}
\right.
$$

\subsection{Sequences of order $3$}

Consider the sequence A005246 satisfying the recurrence
\footnote{Note that, unlike the Somos sequences, this sequence 
also satisfies a linear recurrence: 
$a_{n+4}=4a_{n+2}-a_n.$}

$$
a_{n+3}a_n=a_{n+2}a_{n+1}+1,
$$
and starting as follows:
$
a_n=1, 1, 1, 2, 3, 7, 11, 26, 41, 97, 153, 362, 571, 1351, \ldots
$
This sequence  can be generated by the period~$1$ quivers
$$
a)
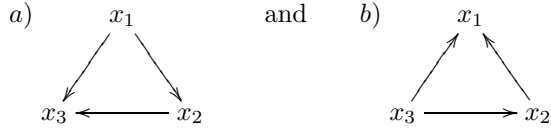
\xymatrix @!0 @R=1.3cm @C=0.9cm
 {
&x_1\ar@{->}[ld]\ar@{->}[rd]&\\
x_3&&x_2\ar@{->}[ll]
}\qquad\hbox{and}\qquad
b) \xymatrix @!0 @R=1.3cm @C=0.9cm
 {
&x_1\ar@{<-}[ld]\ar@{<-}[rd]&\\
x_3&&x_2\ar@{<-}[ll]
}
$$
already considered in Example~\ref{FiEx}.

By Theorem~\ref{Main}, the first quiver has a period~$1$
weight function, but not the second one.
Therefore, the sequence $(b_n)_{n\in\N}$ defined by the recurrence
$$
A_{n+3}A_n=A_{n+2}A_{n+1}+1+w\,\e
$$
and arbitrary integer initial conditions is integer.
Our results do not give any information about the sequence $(b_n)_{n\in\N}$
satisfying
$$
A_{n+3}A_n=A_{n+2}A_{n+1}\left(1+\e\right)+1,
$$
but numerical experiments show that it is not integer.

However, consider again the quiver b).
The weight function $w(x_1)=w(x_2)=w(x_3)=1$ is of period~$6$.
Indeed, after three consecutive mutations $\mu_3\circ\mu_2\circ\mu_1$,
the function $w$ changes its sign and becomes
$w(x'_1)=w(x'_2)=w(x'_3)=-1$, while the quiver remains unchanged.
Therefore, the recurrence
\begin{equation}
\label{Alt3}
A_{n+3}A_n=A_{n+2}A_{n+1}\left(1+(-1)^\frac{(n+1)(n+2)(n+3)}{6}\e\right)+1
\end{equation}
defines integer sequences $(b_n)_{n\in\N}$.
Note that the exponent is chosen to obtain the sign sequence
$+,+,+,-,-,-,+,+,+,\ldots$
Written more explicitly, $(b_n)_{n\in\N}$ satisfies the non-linear recurrence
$$
b_{n+3}a_n=b_{n+2}a_{n+1}+b_{n+1}a_{n+2}-b_na_{n+3}+
(-1)^\frac{(n+1)(n+2)(n+3)}{6}a_{n+2}a_{n+1}.
$$
For example, zero initial conditions lead to the following sequence
$$
b_n=0,0,0,1,3,15,17,43,2,112,84,\ldots
$$

\subsection{Non-homogeneous Somos-$4$ sequence}\label{GenS4Sect}
Let $p$ and $q$ be positive integers,
and consider the sequence $(a_n)_{n\in\N}$ defined by the recurrence
$$
a_{n+4}a_n=a_{n+3}^pa_{n+1}^p+a_{n+2}^q
$$
and the initial conditions $a_0=a_1=a_2=a_3=1$.
This sequence was considered in~\cite{Gal}; see also~\cite{FM}.

\begin{cor}
\label{SomExThm}
(i) The sequence $(b_n)_{n\in\N}$, defined by the recurrence
\begin{equation}
\label{SSecGen1}
A_{n+4}A_n=A_{n+3}A_{n+1}+A_{n+2}^q\left(1+\e\right)
\end{equation}
with arbitrary integer initial conditions $(b_1,b_2,b_3,b_4)$,
is integer.

(ii) The sequence $(b_n)_{n\in\N}$, defined by the recurrence
\begin{equation}
\label{SSecGen2}
A_{n+4}A_n=A_{n+3}^pA_{n+1}^p\left(1+\e\right)+A_{n+2}
\end{equation}
with arbitrary integer initial conditions $(b_1,b_2,b_3,b_4)$,
is integer.
\end{cor}

\begin{proof}
Following~\cite{FM}, consider the quivers
(that differ only by orientation):
\begin{equation}
\label{SomQui}
 \xymatrix@!0 @R=2cm @C=2cm
 {
x_4\ar@{->}[rd]^<<<<<<q\ar@{<-}[d]_p&
x_1\ar@{<-}[ld]^>>>>>>>q\ar@{->}[d]^p\ar@{->}[l]_p\\
x_3&x_2\ar@{->}[l]^{p(q+1)}
}
\qquad\hbox{and}\qquad
 \xymatrix@!0 @R=2cm @C=2cm
 {
x_4\ar@{<-}[rd]^<<<<<<q\ar@{->}[d]_p&
x_1\ar@{->}[ld]^>>>>>>>q\ar@{<-}[d]^p\ar@{<-}[l]_p\\
x_3&x_2\ar@{<-}[l]^{p(q+1)}
}
\end{equation}
where the labels $p,q$ and $p(q+1)$ stand for the number of arrows.
Each of them rotates under the series of consecutive mutations
$\mu_1,\mu_2,\mu_3,\ldots$
For instance,
$$
 \xymatrix@!0 @R=2cm @C=2cm
 {
x_4\ar@{->}[rd]^<<<<<<q\ar@{<-}[d]_p&
x_1\ar@{<-}[ld]^>>>>>>>q\ar@{->}[d]^p\ar@{->}[l]_p\\
x_3&x_2\ar@{->}[l]^{p(q+1)}
}
\qquad
\stackrel{\mu_1}{\Longrightarrow}
\qquad
 \xymatrix@!0 @R=2cm @C=2cm
 {
x_4\ar@{->}[rd]^<<<<<<q\ar@{<-}[d]_{p(q+1)}&
x_1'\ar@{->}[ld]^>>>>>>>q\ar@{<-}[d]^p\ar@{<-}[l]_p\\
x_3&x_2\ar@{->}[l]^{p}
}
$$
This is straightforward from the definition of quiver mutations
(and similarly for the twin quiver).

By Theorem~\ref{Main}, the first of the quivers~(\ref{SomQui})
has a weight function of period~$1$, if (and only if) $p=1$,
while the second quiver has a weight function of period~$1$, if (and only if) $q=2$.
\end{proof}

Note that our results do not imply the converse statement.
However, we conjecture that the sequence $(b_n)_{n\in\N}$, defined by the
recurrence $A_{n+4}A_n=A_{n+3}^pA_{n+1}^p+A_{n+2}^q\left(1+\e\right)$,
is integer {\it if and only if} $p=1$ (and similarly for the second case).
This conjecture is confirmed by the following examples.

\begin{ex}
Let us now consider the sequence~$(A_n)_{n\in\N}$ 
satisfying the recurrence
$A_{n+4}A_n=A_{n+3}^pA_{n+1}^p\left(1+\e\right)+A_{n+2}^q$,
with initial conditions: $b_0=b_1=b_2=b_3=0$,
and take~$q\not=2$.
Although Theorem~\ref{Main} does not imply non-integrality of~$(b_n)_{n\in\N}$,
this sequence is not integer in all the examples we considered.

a) If $q=0$, then
the sequence starts as follows: 
$$
\begin{array}{r|r|c|c|c|c|c|c|c|c|cc}
n&0&1&2&3&4&5&6&7&8&9&\cdots\\[2pt]
\hline
a_n&1&1&1&1&2&3&4&9&14&19&\cdots\\[2pt]
\hline
b_n&0&0&0&0&1&3&6&24&56&\frac{307}{3}&\cdots
\end{array}
$$

b) If $q=1$,
then the sequence stops to be integer one step earlier: 
$$
\begin{array}{r|r|c|c|c|c|c|c|c|cc}
n&0&1&2&3&4&5&6&7&8&\cdots\\[2pt]
\hline
a_n&1&1&1&1&2&3&5&13&22&\cdots\\[2pt]
\hline
b_n&0&0&0&0&1&3&7&32&\frac{159}{2}&\cdots
\end{array}
$$

c) If $q=3$,
then, the sequence is: 
$$
\begin{array}{r|r|c|c|c|c|c|c|c|cc}
n&0&1&2&3&4&5&6&7&8&\cdots\\[2pt]
\hline
a_n&1&1&1&1&2&3&11&49&739&\cdots\\[2pt]
\hline
b_n&0&0&0&0&1&3&18&150&\frac{6539}{2}&\cdots
\end{array}
$$
\end{ex}

This and many other experimental computations
illustrate a sophisticated and fragile nature
of the Laurent phenomenon of Theorem~\ref{LeurThm}.
It seems to occur only when there is a weighted function with period~$1$.

\subsection{Conclusion and an open problem}
The properties of the constructed integer sequences remain unexplored.
In many cases, we cannot prove their positivity 
(although this was checked numerically for the most interesting examples), and
their asymptotics are unknown.

A very interesting property of the Somos-type sequences is their relation to
discrete integrable systems; see~\cite{FH} and references therein.
(The properties of ``integrality'' and ``integrability'' are related not only phonetically!)
It will be interesting to investigate integrability of discrete dynamical systems
related to the sequences constructed in this paper.
For example, is the following map on~$\R^8$ 
$$
\left(
\begin{array}{c}
x_1\\[2pt]
x_2\\[2pt]
x_3\\[2pt]
x_4\\[2pt]
y_1\\[2pt]
y_2\\[2pt]
y_3\\[2pt]
y_4
\end{array}
\right)\longmapsto
\left(
\begin{array}{c}
x_2\\[2pt]
x_3\\[2pt]
x_4\\[2pt]
\left(x_4x_2+x_3^2\right)/x_1\\[2pt]
y_2\\[2pt]
y_3\\[2pt]
y_4\\[4pt]
\left(x_2y_4+2x_3y_3+x_4y_2+x_3^2\right)/x_1-
y_1\left(x_4x_2+x_3^2\right)/x^2_1
\end{array}
\right)
$$
 that arises from the extended
Somos-4 recurrence (see Section~\ref{S4Sect}) completely integrable?

\bigskip
{\bf Acknowledgements}.
This paper was initiated by discussions with Michael Somos; we are
indebted to him for many fruitful comments and a computer program.
We are grateful to Gregg Musiker and Michael Shapiro for enlightening discussions.
This paper was completed when the first author was a Shapiro visiting professor
at Pennsylvania State University, V.O. is grateful to Penn State for its hospitality.
S.T. was partially supported by the NSF Grant DMS-1510055.


\begin{thebibliography}{99}

\bibitem{BPW}
M.Bousquet-M\'elou, J.Propp, J.West.
{\it Perfect matchings for the three-term Gale-Robinson sequences}, 
Electron. J. Combin. {\bf 16:1} (2009), 1--37.

\bibitem{Gal2}
J.A. Cruz Morales, S. Galkin,
{\it Upper bounds for mutations of potentials},
SIGMA Symmetry Integrability Geom. Methods Appl. {\bf 9} (2013), Paper 005, 13 pp. 

\bibitem{FZ1}
S. Fomin, A. Zelevinsky, 
{\it Cluster algebras. I. Foundations. }
J. Amer. Math. Soc.  {\bf 15}  (2002),   497--529.

\bibitem{FZ2}
S. Fomin, A. Zelevinsky, 
{\it The Laurent phenomenon.}
Adv. in Appl. Math.  {\bf 28}  (2002),   119--144.

\bibitem{FH}
A. Fordy, A. Hone, 
{\it Discrete integrable systems and Poisson algebras from cluster maps},
Comm. Math. Phys. {\bf 325} (2014), 527--584.

\bibitem{FM}
A. Fordy, R. Marsh, 
{\it Cluster mutation-periodic quivers and associated Laurent sequences},
J. Algebraic Combin. {\bf 34} (2011), 19--66.

\bibitem{Gal}
D.Gale, 
{\it The strange and surprising saga of the Somos sequences.} 
Math. Intell. {\bf 13,} (1991) 40--42.

\bibitem{Gal1}
S. Galkin, A. Usnich,
{\it Mutations of potentials}, Preprint IPMU 10-0100, 2010.

\bibitem{GHK}
M. Gross, P. Hacking, S. Keel, 
{\it Birational geometry of cluster algebras}, 
Algebr. Geom. {\bf 2} (2015), 137--175.

\bibitem{Mar}
R. Marsh, 
Lecture notes on cluster algebras, Zurich Lectures in Advanced Mathematics. 
European Mathematical Society (EMS), Z\"urich, 2013.

\bibitem{Mor} 
S. Morier-Genoud,
{\it Coxeter's frieze patterns at the crossroads of algebra, geometry and combinatorics},
Bull. Lond. Math. Soc. {\bf 47} (2015), 895--938.

\bibitem{SFriZ} 
S. Morier-Genoud, V. Ovsienko, S. Tabachnikov. 
{\it Introducing supersymmetric frieze patterns and linear difference operators}, 
Math. Z. {\bf 281} (2015), 1061--1087.

\bibitem{SCl} 
V. Ovsienko, 
{\it A step towards cluster superalgebras},
arXiv:1503.01894.

\bibitem{Spe} 
D. Speyer, 
{\it Perfect matchings and the octahedron recurrence},
J. Algebraic Combin. {\bf 25} (2007), 309--348.

\bibitem{Stu} 
E. Study, {Geometrie der Dynamen}. Leipzig, 1903.

\end{thebibliography}
\end{document}